\documentclass[oneside,11pt]{amsart}
\usepackage{amsfonts,amssymb,amscd,amsmath,enumerate,verbatim,calc,graphicx,geometry}
\usepackage[all]{xy}
\newtheorem{theorem}{Theorem}[section]
\newtheorem{lemma}[theorem]{Lemma}
\newtheorem{proposition}[theorem]{Proposition}
\newtheorem{corollary}[theorem]{Corollary}
\theoremstyle{definition}
\theoremstyle{definitions}
\newtheorem{definition}[theorem]{Definition}

\newtheorem{remark}[theorem]{Remark}
\newtheorem{example}[theorem]{Example}
\theoremstyle{notations}

\theoremstyle{remarks}

\newcommand{\sub}{\subseteq}

\newcommand{\lo}{\longrightarrow}

\newcommand{\wt}{\widetilde}
\newcommand{\vf}{\varphi}
\newcommand{\fr}{\frac}
\newcommand{\al}{\alpha}
\newcommand{\la}{\lambda}

\newcommand{\bt}{\beta}
\newcommand{\ti}{\tilde}
\newcommand{\tx}{\textit}
\newcommand{\sq}{\simeq}
\newcommand{\lk}{\langle}
\newcommand{\rg}{\rangle}
\begin{document}
\author[A. Pakdaman , H. Torabi and B. Mashayekhy]
{Ali~Pakdaman,$\ \ $Hamid~Torabi and Behrooz~Mashayekhy $^*$ }

\title[Small loop spaces and covering theory of ...]
{small loop spaces and covering theory of non-homotopically Hausdorff spaces}
\thanks{$^*$Corresponding author}
\thanks{E-mail addresses: alipaky@yahoo.com; hamid$_{-}$torabi86@yahoo.com and bmashf@um.ac.ir}
\maketitle

\begin{center}
{\it Department of Pure Mathematics, Center of Excellence in Analysis on Algebraic Structures, Ferdowsi University of Mashhad,\\
P.O.Box 1159-91775, Mashhad, Iran.}
\end{center}

\vspace{0.4cm}
\begin{abstract}
In this paper we devote to spaces that are not homotopically hausdorff and study their covering spaces. We introduce the notion of small covering and prove that every small covering of $X$ is the universal covering in categorical sense. Also, we introduce the notion of semi-locally small loop space which is the necessary and sufficient condition for existence of universal cover for non-homotopically hausdorff spaces, equivalently existence of small covering spaces. Also, we prove that for semi-locally small loop spaces, $X$ is a small loop space if and only if every cover of $X$ is trivial if and only if $\pi_1^{top}(X)$ is an indiscrete topological group.
\end{abstract}
\vspace{0.5cm}
\section{Introduction and Motivation}
Recall that a continuous map $p:\wt{X}\lo X$ is called a $\textit {covering}$ of $X$, and $\wt{X}$ is called a $\textit {covering space}$ of $X$, if for every $x\in X$ there is an open
subset $U$ of $X$ with $x\in U$ such that $U$ is $\tx{evenly covered}$ by $p$, that is,
$p^{-1}(U)$ is a disjoint union of open subsets of $\wt{X}$ each of which is mapped
homeomorphically onto $U$ by $p$.

In the classical covering theory, one assumes that $X$ is, in addition, connected, locally path connected and wishes to classify all path connected covering spaces of $X$
and to find among them the $\tx{universal covering space}$ in categorical sense, that is, a covering
$p:\wt{X}\lo X$ with the property that for every covering $q:\wt{Y}\lo X$ by a path connected space $\wt{Y}$ there is a covering $f:\wt{X}\lo\wt{Y}$ such that $q\circ f= p$. If $X$
is locally path connected, we have the following well-known result which
can be found for example in \cite{S}:

\emph{Every simply connected covering space of $X$ is a universal covering space.
Moreover, X admits a simply connected covering space if and only if $X$
is semi-locally simply connected, in which case the coverings $p:(\wt{X},\ti{x})\lo(X,x)$
with path-connected $\wt{X}$ are in direct correspondence with the conjugacy classes of
subgroups of $\pi_1(X,x)$ via the monomorphism $p_*:\pi_1(\wt{X},\ti{x})\lo\pi_1(X,x)$.}

Outside of locally nice spaces, the traditional theory of covering is not as
pleasant. Although H. Fischer and A. Zastrow in \cite{FZ} studied universal covers of homotopically hausdorff spaces.
They defined generalized regular coverings of topological space $X$ as functions
$p:\wt{X}\loX$ satisfying the following conditions for some normal subgroup $H$ of
$\pi_1(X)$:\\
$U_1$. $\wt{X}$ is a connected, locally path connected space.\\
$U_2$. The map $p:\wt{X}\lo X$ is a continuous surjection and $\pi_1(p):\pi_1(\bar{X})\lo\pi_1(X)$
is a monomorphism onto $H$.\\
$U_3$. For every connected, locally path connected space $Y$, for every continuous function $f:(Y,y)\lo(X, x)$
with $f_*(\pi_1(Y, y))\sub H$, and for every $\ti{x}\in \ti{X}$ with $p(\ti{x})=x$, there is a
unique continuous $g:(Y, y)\lo (\ti{X},\ti{x})$ with $p\circ g=f$.

This generalized notion of universal covering $p:\wt{X}\lo X$ enjoys
most of the usual properties, with the possible exception of
evenly covered neighborhoods. If $X$ is path connected, locally path-connected and semi-locally simply connected,
then $p:\wt{X}\lo X$ agrees with the classical universal covering.
A necessary condition for the standard construction to yield a
generalized universal covering is that X be homotopically Hausdorff. A space $X$ is $\textit{homotopically Hausdorff}$ if given any point $x$ in $X$ and any nontrivial homotopy class
$[\al]\in\pi_1(X,x)$, there is a neighborhood $U$ of $x$ which contains no representative for $[\al]$.
Accordingly, we would like to
study coverings of non-homotopically Hausdorff spaces and discover the topology type of their
universal covering spaces.
\begin{definition} (\cite{V})
A loop $\al:(S^1, 0)\lo(X,x)$ is $small$ if and only if there exists a representative of the homotopy class $[\al]\in\pi_1(X,x)$
in every open neighborhood $U$ of $x$. A $small\ loop$ is nontrivial small loop if it is not homotopically trivial. The $\tx{small loop
group}$ $\pi_1^s(X,x)$ of $(X,x)$ is the subgroup of the fundamental group $\pi_1(X,x)$ consisting of homotopy classes of small
loops. A non-simply connected space X is called {\it small loop space} if for every $x\in X$, every loop $\al:(I,\partial I)\lo (X,x)$ is small.
\end{definition}
\begin{definition}The SG subgroup of $\pi_1(X, x)$ is denoted by $\pi_1^{sg}(X, x)$ and is generated by $\{[\al*\bt*\al^{-1}]\ |\ [\bt]\in\pi_1^s(X,\al(1))\}$ for which $\al$ is a path with $\al(0)=x$.
\end{definition}
It is easy to check that the small loop group is a subgroup, but not necessarily a normal subgroup of $\pi_1(X,x)$. Furthermore,
it is a functor from $hTop_*$ to $Groups$. The presence of small loops implies the absence of semi-locally simply connectedness and homotopically Hausdorffness. While general non-semi-locally simply connected spaces may have non-homotopic nontrivial loops at every neighborhood of some point, in the case of a small loop the homotopy type of one loop can be chosen for all neighborhoods of the base point.

In section 2, we find out some relation between small loop groups, SG subgroups and fundamental groups.
Z. Virk in \cite{V} constructed a small loop space by using Harmonic Archipelago and studied the relation between covering spaces and small loop group, for example, for the universal covering space $p:\wt{X}\lo X$ of non-homotopically Hausdorff space $X$, $p_*\pi_1(\wt{X},\ti{x})$ contains $\pi_1^{sg}(X,x)$ as a subgroup and so $X$ has no simply connected cover.  We show that this holds for every cover of $X$. In section 3, we show that if $\wt{X}$ is a small loop space, then $\wt{X}$ is a universal cover of $X$. Also, in this case $X$ is not homotopically Hausdorff and $\pi_1^s(X,x)=\pi_1^{sg}(X,x)=p_*\pi_1(\wt{X},\ti{x})$ which is normal in $\pi_1(X,x)$.

Finally, in section 4, we present the main result of this article which state that a space $X$ has a small loop space as cover if and only if $X$ is a semi-locally small loop space, that is, for every $x\in X$ there exists an open neighborhood $U$ such that $i_*\pi_1(U,y)=\pi_1^s(X,y)$, for all $y\in U$, where $i:U\lo X$ is the inclusion. Moreover, we show that if $X$ is a semi-locally small loop space, then $X$ is a small loop space if and only if every cover of $X$ is trivial if and only if the topological fundamental group of $X$ is indiscrete.

throughout this article, all the homotopies between two paths are relative to end points.
\section{small loop groups}
The importance of small loops in the covering space theory was pointed out by Brodskiy,
Dydak, Labuz, and Mitra in \cite{BD1} and \cite{BD2} and by Virk in \cite{V}.
In this section, we study basic properties of small loops, small loop group and SG subgroup of fundamental group of non-homotopically Hausdorff spaces $X$ and their relations to the covering spaces of $X$.
It is obvious that if the topological space $X$ is not homotopically hausdorff then there exists $x\in X$ such that $\pi_1^s(X,x)\neq 1$.
\begin{theorem}
For every cover $p:\wt{X}\lo X$ and $x\in X$ we have $\pi_1^s(X,x)\leq\pi_1^{sg}(X,x)\leq p_*\pi_1(\wt{X},\tilde{x})$.
\end{theorem}
\begin{proof}
By definition $\pi_1^s(X,x)\leq\pi_1^{sg}(X,x)$. Let $[\al]\in \pi_1^s(X,x)$ and let $U$ be an evenly covered open neighborhood of $x$ such that $p^{-1}(U)=\cup_{j\in J}V_j$. Since $\al$ is a small loop, there exists a loop $\al_U:I\lo U$ such that $[\al]=[\al_U]$ which implies existence of the loop $\wt{\al_j}=(p|_{V_j})^{-1}\circ\al_U$. Therefore $[\al]=p_*([\wt{\al_j}])$ and hence $\pi_1^s(X,x)\sub p_*\pi_1(\wt{X},\tilde{x})$.
Let $[\al*\bt*\al^{-1}]\in\pi_1^{sg}(X,x)$, where $[\bt]\in \pi_1^s(X,\al(1))$. If $\wt{\al}$ is the lift of $\al$ with initial point $\ti{x}$ and end point $\ti{y}$, then $y:=p(\ti{y})=\al(1)$. Since $\pi_1^s(X,y)\leq p_*\pi_1(\wt{X},\tilde{y})$, we can assume $p_*([\wt{\bt}])=[\bt]$ for some $[\wt{\bt}]\in\pi_1(\wt{X},\ti{y})$ and then $p_*([\wt{\al}*\wt{\bt}*\wt{\al}^{-1}])=[\al*\bt*\al^{-1}]$ which implies that $\pi_1^{sg}(X,x)\leq p_*\pi_1(\wt{X},\tilde{x})$.
\end{proof}
\begin{corollary}
If there exists $x\in X$ such that $\pi_1^s(X,x)\neq 1$, then $X$ does not admit a simply connected cover.
\end{corollary}
\begin{corollary}
Let $X$ be a connected, simply connected, locally path connected space. If
$G$ is a group acting properly on $X$, then $X/G$ has no small loop.
\end{corollary}
\begin{corollary}
For a cover $p:\wt{X}\lo X$, $\pi_1^{sg}(X,x)$ acts trivially on $p^{-1}(\{x\})$, that is, $\ti{x}.[\al]=\ti{x}$ for $\ti{x}\in p^{-1}(\{x\})$ and $[\al]\in \pi_1^{sg}(X,x)$.
\end{corollary}
\begin{example}
Fischer and Zastrow \cite[Examples 4.13, 4.14, 4.16]{FZ} introduced some spaces that admit generalized universal cover. Since for the existence of generalized universal covers, the condition homotopically Hausdorff is necessary, spaces which admit generalized universal cover has no small loop. For example, subsets of closed surfaces, 1-dimensional compact Hausdorff spaces, 1-dimensional separable metrizable spaces, and trees of manifolds.
\end{example}
\begin{remark}
If $X$ is small loop space, then $\pi_1^s(X,x)=\pi_1(X,x)$ and since $\pi_1^s(X,x)\leq\pi_1^{sg}(X,x)\leq\pi_1(X,x)$, for small loop spaces we have $\pi_1^s(X,x)=\pi_1^{sg}(X,x)=\pi_1(X,x)$.
Note that if $\pi_1^s(X,x)=\pi_1^{sg}(X,x)$, it is not necessary that $\pi_1^s(X,x)=\pi_1^{sg}(X,x)=\pi_1(X,x)$. For example, consider a non-simply connected, semi-locally simply connected space.
\end{remark}
\section{small loop spaces and covers}
In keeping with modern nomenclature, the term universal covering
space will always mean a categorical universal object, that is, a covering
$p:\wt{X}\lo X$ with the property that for every covering $q:\wt{Y}\lo X$ by a path connected space $\wt{Y}$ there is a covering $f:\wt{X}\lo\wt{Y}$ such that $q\circ f= p$. The following theorem comes from definitions.
\begin{theorem}
Every cover of a small loop space $X$ is homeomorphic to $X$.
\end{theorem}
\begin{proof}
Let $p:\wt{X}\lo X$ be a cover of small loop space $X$. Then by Theorem 2.1 $\pi_1^s(X,x)\leq p_*\pi_1(\wt{X},\tilde{x})\leq\pi_1(X,x)=\pi_1^s(X,x)$ , for each $x\in X$ which implies that $\wt{X}$ is a one sheeted cover of $X$ and so they are homeomorphic.
\end{proof}
\begin{definition}
By a $\tx{small covering}$ of a topological space $X$ we mean a covering $p:\wt{X}\lo X$ such that $\wt{X}$ is a small loop space.
\end{definition}
\begin{lemma}
If $X$ has a small covering $p:\wt{X}\lo X$, then a loop $\al$ in $X$ is small if and only if there exists a lift $\wt{\al}$ of $\al$ in $\wt{X}$ which is closed.
\end{lemma}
\begin{proof}
Let $\al$ be a small loop at $x$ in $X$, $U$ be an evenly covered open subset of $X$ containing $x$ and $\al_U$ be a loop at $x$ in $U$ that is homotopic to $\al$. If $\ti{x}\in p^{-1}(\{x\})$ and $V$ be the homeomorphic copy of $U$ in $p^{-1}(U)$ which contains $\ti{x}$, then $\wt{\al_U}=(p|_{V})^{-1}\circ\al_U$ is a loop at $\ti{x}$ in $\wt{X}$ which easily seen that it is a small loop. Since $\al\simeq\al_U$, if $\wt{\al}$ is the lift of $\al$ with initial point $\ti{x}$, then $\wt{\al}(1)=\wt{\al_U}(1)=\ti{x}$ which implies that $\al$ has closed lift. Conversely, assume $\al$ has closed lift $\wt{\al}$ and $U,V$ be chosen as above. Let $\wt{\al}_V$ be a loop at $\ti{x}$ in $V$ that is homotopic to $\wt{\al}$, then $\al_U=p\circ \wt{\al}_V$ is a loop at $x$ and homotopic to $\al$. Since every open subset of an evenly covered open subset is evenly covered, the result holds.
\end{proof}
\begin{corollary}
If $X$ has a small covering $p:\wt{X}\lo X$, then $X$ is not homotopically Hausdorff.
\end{corollary}
\begin{theorem}
If $X$ has a small covering $p:\wt{X}\lo X$, then $p_*\pi_1(\wt{X},\ti{x})$ is a normal subgroup of $\pi_1(X,x)$.
\end{theorem}
\begin{proof}
Let $[\al]\in p_*\pi_1(\wt{X},\ti{x})$ and $[\bt]\in\pi_1(X,x)$. There exists a closed lift $\wt{\al}$ of $\al$. Let $\wt{\bt}$ be the lift of $\bt$ with initial point $\ti{x}$. If $\gamma$ is the lift of $\bt^{-1}$ with initial point $\wt{\bt}(1)$, then $\gamma=\wt{\bt}^{-1}$. Hence $\wt{\bt}*\wt{\al}*\wt{\bt}^{-1}$ is a loop and $p\circ(\wt{\bt}*\wt{\al}*\wt{\bt}^{-1})=\bt*\al*\bt^{-1}$ which implies that $\bt*\al*\bt^{-1}$ is small since it has a closed lift. Therefore $p_*\pi_1(\wt{X},\ti{x})$ is a normal subgroup of $\pi_1(X,x)$.
\end{proof}

\begin{proposition}
If $X$ has a small covering $p:\wt{X}\lo X$, then $\pi_1^s(X,x)=\pi_1^{sg}(X,x)=p_*\pi_1(\wt{X},\ti{x})$.
\end{proposition}
\begin{proof}
Assume $[\al]=[p\circ\bt]$ for $[\bt]\in\pi_1(\wt{X},\ti{x})$, $U$ is an evenly covered open subset of $X$ containing $x$ and $V$ is the homeomorphic copy of $U$ in $\wt{X}$ containing $\ti{x}$. Since $\wt{X}$ is a small loop space, there exists a loop $\bt_V$ in $V$ such that $\bt\sq\bt_V$. Hence $\al\sq\al_U:=p\circ\bt_V$. Therefore $p_*\pi_1(\wt{X},\ti{x})\leq\pi_1^s(X,x)$ and by Theorem 2.1 the result
\end{proof}
\begin{corollary}
If $X$ has a small covering, then $\pi_1^s(X,x)$ is a normal subgroup of $\pi_1(X,x)$, for every $x\in X$.
\end{corollary}
We denote $\mathcal{COV}$(X) to be the category of all coverings of $X$ as objects and covering maps between them as morphisms.
\begin{theorem}
A small covering of space $X$ is the universal object in the category $\mathcal{COV}$(X).
\end{theorem}
\begin{proof}
Assume $p:\wt{X}\lo X$ is a small cover of $X$ and $q:\wt{Y}\lo X$ is another covering, then by Proposition 3.6 $p_*\pi_1(\wt{X})=\pi_1^{sg}(X)\leq q_*\pi_1(\wt{Y})$ which implies that there exists $f:\wt{X}\lo\wt{Y}$ such that $q\circ f=p$.
\end{proof}
\section{Existence}
We now consider the following natural question: \emph{Given a non-homotopically Hausdorff space $X$, does
$X$ have a small covering space?} First, we derive a rather simple necessary
condition. Let $(\wt{X}, p)$ be a small covering space of $X$, $x$ be an arbitrary
point of $X$, $\ti{x}\in p^{-1}(\{x\})$, $U$ be an evenly covered open neighborhood of $x$
and let $V$ be the homeomorphic copy of $U$ in $p^{-1}(U)$ which contains the point $\ti{x}$. We then
have the following commutative diagram involving fundamental groups:
$$\xymatrix{
\pi_1(V,\ti{x}) \ar[r]^{j_*} \ar[d]_{(p|_V)_*}
& \pi_1(\wt{X},\ti{x})\ar[d]^{p_*}  \\
\pi_1(U,x) \ar[r]^{i_*} & \pi_1(X,x).  }$$

Obviously $\pi_1^s(X,x)\sub i_*\pi_1(U,x)$. Let $\al$ be a loop at $x$ in $U$ such that $i_*([\al])\neq 1$ and let $\wt{\al}=p|_V^{-1}\circ\al$. Since $(p|_V)_*$ is isomorphism and $p_*$ is injective, commutativity of diagram implies that $j_*([\wt{\al}])\neq 1$. Therefore $\wt{\al}$ is a nontrivial small loop in $\wt{X}$ and since $p_*\pi_1(\wt{X},\ti{x})=\pi_1^s(X,x)$, $[\al]\in\pi_1^s(X,x)$. Thus we conclude that the space $X$ has the following
property: Every point $x\in X$ has a neighborhood $U$ such that
$i_*\pi_1(U,x)=\pi_1^s(X,x)$. We call a space with this property a semi-locally
small loop space. Since for every $y\in X$, we have $i_*\pi_1(U,y)=\pi_1^s(X,y)$, this definition can also be rephrased as follows:
\begin{definition}
A space $X$ is a semi-locally small loop space if and only if for each $x\in X$ there exists an open neighborhood $U$ of $x$ such that $i_*\pi_1(U,y)=\pi_1^s(X,y)$, for all $y\in U$, where $i:U\lo X$ is the inclusion.
\end{definition}
In the sequel, all the spaces are path connected.
\begin{lemma}
If $X$ is a semi-locally small loop space, then $\pi_1^{sg}(X,x)=\pi_1^s(X,x)$, for every $x\in X$. Also $\pi_1^s(X,x)\cong\pi_1^s(X,y)$, for every $x,y\in X$.
\end{lemma}
\begin{proof}
For every $x\in X$ we have $\pi_1^s(X,x)\leq\pi_1^{sg}(X,x)$. For converse, assume that $[\al*\bt*\al^{-1}]\in\pi_1^{sg}(X,x)$, where $[\bt]\in \pi_1^s(X,\al(1))$.

For every $t\in I$, let $U_t$ be an open neighborhood of $\al(t)$ such that $i_*\pi_1(U_t,y)=\pi_1^s(X,y)$, for all $y\in U_t$. The sets $\al^{-1}(U_t)$, $t\in I$, form an open cover of $I$. Let $\la>0$ be the Lebesgue number for this cover. Choose $N\in\mathbb{N}$ so that $1/N<\la$. For
each $1<n\leq N$ let $$I_n=[\fr{n-1}{N},\fr{n}{N}]\sub I.$$
Reindex the $U_t$'s so that
$$\al(I_n)\sub U_n\ \text{for each}\  1<n\leq N.$$
The $U_n$'s are not necessarily distinct, nor does the proof require this condition.
For each $1\leq n\leq N$, denote $y_n=\al(\fr{n}{N})$, $y_0=x$ and $\al_n=\al|_{I_n}\circ\vf_n$, where $\vf_n:I\lo I_n$ is the linear homeomorphism. Since $\bt\in\pi_1^s(X,y_N)=i_*\pi_1(U_N,y_N)$, there exists a loop $\bt_N:I\lo U_N$ such that $\bt_N\sq\bt$. Then ${\al_N}*\bt_N*{\al_N}^{-1}$ is a loop in $U_N$ at $y_{N-1}\in U_{N-1}\cap U_N$. Hence there exists a loop $\bt_{N-1}:I\lo U_{N-1}\cap U_N$ such that $\bt_{N-1}\sq{\al_N}*\bt_N*{\al_N}^{-1}$ since $i_*\pi_1(U_N,y_{N-1})=\pi_1^s(X,y_{N-1})$. Similarly, for every $1\leq n\leq N$, there is $\bt_n:I\lo U_n\cap U_{n+1}$ such that $\bt_n\sq{\al_{n+1}}*\bt_{n+1}*{\al_{n+1}}^{-1}$. Therefore we have
$$ \al*\bt*\al^{-1}\sq{\al_1}*{\al_2}*...*{\al_N}*\bt*{\al_N}^{-1}*...*{\al_1}^{-1}$$
$$\sq{\al_1}*{\al_2}*...*{\al_N}*\bt_N*{\al_N}^{-1}*...*{\al_1}^{-1}$$
$$\sq{\al_1}*...*{\al_{N-1}}*\bt_{N-1}*{\al_{N-1}}^{-1}*...*{\al_1}^{-1}\sq...\sq{\al_1}*\bt_1*{\al_1}^{-1}.$$
Since $i_*\pi_1(U_1,x)=\pi_1^s(X,x)$, there exists a loop $\al'\in\pi_1^s(X,x)$ such that $\al_1*\bt_1*\al_1^{-1}\sq\al'$ which implies $[\al*\bt*\al^{-1}]\in\pi_1^s(X,x)$ and hence the first assertion holds. Since $X$ is path connected, $\pi_1^{sg}(X,x)\cong\pi_1^{sg}(X,y)$ for every $x,y\in X$ and hence the second assertion holds.
\end{proof}
\begin{lemma}
Let $X$ be a semi-locally small loop space. If for at least one point $x\in X$ there exists a covering $p:\wt{X}\lo X$ such that $p_*\pi_1(\wt{X},\ti{x})=\pi_1^s(X,x)$, then $\wt{X}$ is a small loop space.
\end{lemma}
\begin{proof}
First we show that for every $y\in X$ and $\ti{y}\in p^{-1}(\{y\})$, $p_*\pi_1(\wt{X},\ti{y})=\pi_1^s(X,y)$. Let $\wt{\al}$ be a path from $\ti{x}$ to $\ti{y}$ and $\al=p\circ\wt{\al}$, then $\theta:p_*\pi_1(\wt{X},\ti{y})\lo p_*\pi_1(\wt{X},\ti{x})$ by $\theta([\bt])=[\al*\bt*\al^{-1}]$ is an isomorphism. Also, by Lemma 4.2 $\eta:\pi_1^s(X,x)\lo\pi_1^s(X,y)$ by $\eta([\gamma])=[\al^{-1}*\gamma*\al]$ is an isomorphism. Since $p_*\pi_1(\wt{X},\ti{x})=\pi_1^s(X,x)$, $\eta\circ\theta$ is well defined and $\eta\circ\theta([\bt])=[\bt]$, which implies that $p_*\pi_1(\wt{X},\ti{y})=\pi_1^s(X,y)$, for every $y\in X$.

Now let $\wt{\gamma}:I\lo\wt{X}$ be a loop at $\ti{y}$ and $\wt{U}$ be an open neighborhood of $\ti{y}$, we show that $[\wt{\gamma}]$ has a representation in $\wt{U}$. For this, let $U$ be an evenly covered open neighborhood of $y=p(\ti{y})$ such that $U\sub p(\wt{U})$ and $i_*\pi_1(U,z)=\pi_1^s(X,z)$, for every $z\in U$. Since $\gamma=p\circ\wt{\gamma}\in p_*\pi_1(\wt{X},\ti{y})=\pi_1^s(X,y)$, there exists a loop $\gamma_U:I\lo U$ at $y$ such that $\gamma\sq\gamma_U$. If $V\sub\wt{U}$ is a homeomorphic copy of $U$ in $\wt{X}$ which contains $\ti{y}$, then $\wt{\gamma}_U=p|_V^{-1}\circ\gamma_U$ is a loop at $\ti{y}$ in $V$ and $\wt{\gamma}\sq\wt{\gamma}_U$, as desired.
\end{proof}

Now we are in a position to state and prove the main result of the paper.
\begin{theorem}
A locally path connected topological space $X$ has a small covering space if and only if $X$ is a semi-locally small loop space.
\end{theorem}
\begin{proof}
If $X$ has a small covering space, by the argument at the beginning of this section, the result holds. Conversely, assume $X$ is a semi-locally small loop space.
Choose a base point $x\in X$ and let $P(X,x)$ be the family of all paths $\al$ in $X$ with $\al(0)=x$. We define an equivalent relation $\sim$ on $P(X,x)$ as follows. For any $\al_1,\al_2\in P(X,x)$, $\al_1\sim\al_2$ if and only if $\al_1$ and $\al_2$ have the same end point and $\al_1*\al_2^{-1}$ is a small loop. We denote the equivalence class of $\al$ by $\lk\al\rg$. Define $\wt{X}$ to be the set of all equivalence classes of paths $\al\in P(X,x)$. Define a function
$p:\wt{X}\lo X$ by setting $p(\lk\al\rg)$ equal to the terminal point of the path class $\al$. We
shall now show how to topologize $\wt{X}$ so that it is a small loop space
and $(\wt{X}, p)$ is a covering space of $X$.

Observe that our hypotheses imply that the topology on $X$ has a basis
consisting of open sets $U$ with the following properties: $U$ is path connected
and every loop in $U$ is small.
For brevity let us agree to call such an open set $U$ basic. Note that, if $x$ and
$y$ are any two points in a basic open set $U$, then for any two paths $f$ and $g$ in $U$
with initial point $x$ and terminal point $y$, $f*g^{-1}$ is small loop.

Given any path $\al\in\wt{X}$ and any basic open set $U$ which contains the end
point $p(\lk\al\rg)$, denote by $(\lk\al\rg,U)$ the set of all equivalence classes  $\lk\bt\rg$ such that for some path
class $[\al']$ with $Im(\al')\sub U$, $\bt=\al*\al'$. Then $(\lk\al\rg,U)$ is a subset of $\wt{X}$. We topologize $\wt{X}$ by
choosing the family of all such sets $(\lk\al\rg,U)$ as a basis of open sets. In
order that the family of all sets of the form $(\lk\al\rg,U)$ can be a basis for some
topology on $\wt{X}$, it is enough to show that if $\gamma\in(\lk\al\rg,U)\cap(\lk\bt\rg,V)$, then there exists a basic open set $W$ such that $(\lk\gamma\rg,W)\sub(\lk\al\rg,U)\cap(\lk\bt\rg,V)$.
For this, we choose $W$ to be any basic
open set such that $p(\gamma)\in W\sub U\cap V$.

Before proceeding with the proof that $(\wt{X}, p)$ is a small covering space
of $X$, it is convenient to make the following two observations:\\
(a) Let $\al\in\wt{X}$, and let $U$ be a basic open neighborhood of $p(\lk\al\rg)$. Then $p|_{(\lk\al\rg, U)}$
is a one-to-one map of $(\lk\al\rg, U)$ onto $U$.\\
(b) Let $U$ be any basic open set, and let $y$ be any point of $U$. Then,
$$p^{-1}(U)=\bigcup_{\lambda}(\lk\al_{\lambda}\rg, U),$$
where $\lk\al_{\lambda}\rg$ denotes
the totality of all path classes in $X$ with initial point
$x$ and terminal point $y$. Moreover, the sets $(\lk\al_{\lambda}\rg,U)$ are pairwise disjoint.\\
For proving (a) suppose that $\ti{y},\ti{z}\in(\lk\al\rg,U)$ and $p(\ti{y})=p(\ti{z})$. Now $\ti{z}=\lk\al*\mu\rg$, where $\mu(0)=\al(1)=x_1$ and $\mu(I)\sub U$; similarly, $\ti{y}=\lk\al\eta\rg$, where $\eta(0)=x_1$ and $\eta(I)\sub U$. Since $p(\ti{y})=p(\ti{z})$, we have $\eta(1)=\mu(1)$, so that $\eta*\mu^{-1}$ is a closed path in $U$ at $x_1$. By the choice of $U$, $\eta*\mu^{-1}$ is a small loop in $X$. Hence $[\al*\eta*\mu^{-1}*\al^{-1}]\in\pi_1^{sg}(X,x)$ which implies that $\al*\eta*\mu^{-1}*\al^{-1}$ is a small loop at $x$, since $\pi_1^{sg}(X,x)=\pi_1^s(X,x)$ by lemma 4.2 . Therefore $\lk\al*\eta\rg=\lk\al*\mu\rg$, that is, $\ti{y}=\ti{z}$. Surjectivity of $p|_{(\lk\al\rg, U)}$ follows from path connectedness of $U$.

Clearly, $p^{-1}(U)$ contains $\bigcup_{\lambda}(\lk\al_{\lambda}\rg, U)$. For the reverse inclusion, let $\ti{y}\in \wt{X}$ such that $p(\ti{y})\in U$, that is, $\ti{y}=\lk\al\rg$ and $\al(1)\in U$. Since $U$ is path connected, there is a path $\la$ in $U$ from $\al(1)$ to $y$. Then $\lk\al_{\la}\rg=\lk\al*\la\rg$ lies in the fiber over $y$. Since $\lk\al\rg=\lk\al*\la*\la^{-1}\rg$ and $\lk\al*\la*\la^{-1}\rg\in (\lk\al*\la\rg,U)$, $\lk\al\rg\in (\lk\al_{\la}\rg,U)$. Therefore (b) holds.

Note that it follows from (b) that p is continuous. Hence, $p|_{(\lk\al\rg, U)}$ is a
one-to-one continuous map from $(\lk\al\rg,U)$ onto $U$, by (a). We assert that $p|_{(\lk\al\rg,U)}$ is
an open map of $(\lk\al\rg,U)$ onto $U$. For, any open subset of $(\lk\al\rg,U)$ is a union of sets
of the form $(\lk\bt\rg,V)$, where $V\sub U$, and hence the fact that $p|_{(\lk\al\rg,U)}$ is open also
follows from (a). Thus, $p$ maps $(\lk\al\rg,U)$ homeomorphically onto $U$. Since $U$ is
path connected, so is $(\lk\al\rg,U)$. Because the sets $(\lk\al_{\lambda}\rg, U)$ occurring in statement
(b) are pairwise disjoint, it follows that any basic open set $U\sub X$ has all the
properties required of an evenly covered neighborhood.

Next, we shall prove that the space $\wt{X}$ is path connected. Let $\ti{x}\in\wt{ X}$
denote the equivalence class of the constant path at $x$. Given any point $\lk\al\rg\in\wt{X}$,
it suffices to exhibit a path joining the points $\ti{x}$ and $\lk\al\rg$. For this purpose, choose
a path $f:I\lo X$ belonging to the equivalence class $\lk\al\rg$. For any real number
$s\in I$, define $f_s:I\lo X$ by $f_s(t)=f(st)$, $t\in I$. Then $f_1=f$ and $f_0$ is constant
path at $x$. We assert that
the map $s\mapsto \lk\al_s\rg$ is a continuous map $I\lo \wt{X}$, i.e. a path in $\wt{X}$. To prove this
assertion, we must check that, for any $s_0\in I$ and any basic neighborhood $U$
of $f(s_0)$, there exists a real number $\delta > 0$ such that if $|s - so| < \delta$, then
$\lk\al_s\rg\in (\lk\al_{s_0}\rg,U)$. For this purpose, we choose $\delta$ so that if $|s - so| < \delta$, then $f(s)\in U$; such a number $\delta$ exists because $f$ is continuous. Thus $s\mapsto\lk\al_s\rg$ is a path in
$\wt{X}$ with initial point $\ti{x}$ and terminal point $\lk\al\rg$, as required.

Finally, we must prove that $\wt{X}$ is a small loop space. For this, we first show that $p_*\pi_1(\wt{X}, \ti{x})=\pi_1^s(X,x)$. We know that $p_*\pi_1(\wt{X}, \ti{x})$ is the isotropy subgroup corresponding to the point $\ti{x}$ for the action of $\pi_1(X, x)$ on
$p^{-1}(\{x\})$. By Lemma 2.1 it suffices to prove that $p_*\pi_1(\wt{X}, \ti{x})\leq\pi_1^s(X,x)$. For, let $[\al]\in p_*\pi_1(\wt{X},\ti{x})$ which implies that $\ti{x}.[\al]=\ti{x}$. The path $\wt{\al}:I\lo \wt{X}$ by $\wt{\al}(s)=\lk\al_s\rg$ is a lifting of $\al$. This path in
$\wt{X}$ has $\ti{x}$ as initial point and $\lk\al\rg\in\wt{X}$ as terminal point. Hence, by the definition of the action of $\pi_1(X, x)$ on
$p^{-1}(\{x\})$, $\ti{x}.[\al]=\ti{x}$ if and only if $\ti{x}=\lk\al\rg$ or equivalently $\al$ is a small loop. Hence, $p_*\pi_1(\wt{X}, \ti{x})=\pi_1^s(X,x)$, as required. Now by Lemma 4.3 , $\wt{X}$ is small loop space.
\end{proof}
\begin{theorem}
Suppose $X$ is a locally path-connected, semi-locally
small loop space. Then for every subgroup $H\leq\pi_1(X,x)$ containing $\pi_1^{sg}(X,x)$ there is a covering
space $p:\wt{X}_H\lo X$ such that $p_*\pi_1(\wt{X}_H,\ti{x})=H$ for a suitably chosen base point
$\ti{x}\in \wt{X}_H$.
\end{theorem}
\begin{proof}
 For points $[\al],[\al']$ in the small covering space $\wt{X}$ of $X$ constructed
above, define
$[\al]\sim[\al']$ to mean $\al(1)=\al'(1)$ and $[\al*\al^{-1}]\in H$. It is easy to see
that this is an equivalence relation. Let $\wt{X}_H$ be the quotient space of $\wt{X}$
obtained by identifying $[\al]$ with $[\al']$ if $[\al]\sim[\al']$. Note that if $\al(1)=\al'(1)$, then
$[\al]=[\al']$ if and only if $[\al*\bt]=[\al'*\bt]$. This means that if any two points in basic neighborhoods
$(\langle\al\rg,U)$ and $(\langle\al'\rg,U)$ are identified in $\wt{X}_H$, then the whole neighborhoods are identified. Hence
the natural projection $\wt{X}_H\lo X$ induced by $[\al]\mapsto\al(1)$ is a covering space.
If we choose for the base point $\ti{x}\in \wt{X}_H$ the equivalence class of the constant path
$e_x$, then the image of $p_*:\pi_1(\wt{X}_H,\ti{x})\lo\pi_1(X,x)$ is exactly $H$. This is because
for a loop $\al$ in $X$ based at $x$, its lift to $\wt{X}$ starting at $e_x$ ends at $[\al]$, so the image
of this lifted path in $\wt{X}_H$ is a loop if and only if $[\al]=[e_x]$, or equivalently $[\al]\in H$.
\end{proof}

The authors \cite{P} proved that the topological fundamental group of a small loop space is an indiscrete topological group. Also, in Section 3, it is proved that every cover of a small loop space is trivial. By \cite[Theorem 5.5]{B}, the connected covers of $X$ are classified by
conjugacy classes of open subgroups of $\pi_1^{top}(X,x)$, so if $\pi_1^{top}(X,x)$ is an indiscrete topological group, then every cover of $X$ is trivial. Hence there is an interesting closed relation between the small loop properties of a given space $X$, the topological fundamental group $\pi_1^{top}(X,x)$ and the number of covering space of $X$ as in the following theorem.
\begin{theorem}
If $X$ is a locally path connected, semi-locally small loop space, then the following statements are equivalent\\
(i) $X$ is a small loop space.\\
(ii) $\pi_1^{top}(X,x)$ is an indiscrete topological group.\\
(iii) every cover of $X$ is trivial.
\end{theorem}
\begin{proof}
It suffices to prove $(iii)\rightarrow(i)$. Since $X$ is a semi-locally small loop space, it has a small covering by Theorem 4.4  and therefore $X$ is a small loop space, as desired.
\end{proof}
\subsection*{Acknowledgements}
This research was supported by a grant from Ferdowsi University of Mashhad.

\end{document}